\documentclass[reqno]{amsart}

\input xy
\xyoption{all}
\usepackage{accents}
\usepackage{epsfig}
\usepackage{color}
\usepackage{amsthm}
\usepackage{amssymb}
\usepackage{amsmath}
\usepackage{amscd}
\usepackage{amsopn}
\usepackage{graphicx}
\usepackage{tikz-cd} 

\usepackage{xspace}

\usepackage{hhline}
\usepackage{easybmat}

\usepackage{url}
\usepackage{enumitem, hyperref}\hypersetup{colorlinks}

\usepackage{color} %\textcolor{red}{Test}

\definecolor{darkred}{rgb}{1,0,0} %can change the intensity in [0,1]
\definecolor{darkgreen}{rgb}{0,0.8,0}
\definecolor{darkblue}{rgb}{0,0,1}

\hypersetup{colorlinks,
linkcolor=darkblue,
filecolor=darkgreen,
urlcolor=darkblue,
citecolor=darkgreen}

\makeatletter
\def\reflb#1#2{\begingroup
    #2%
    \def\@currentlabel{#2}%
    \phantomsection\label{#1}\endgroup
}
\makeatother

\numberwithin{equation}{section}
\newtheorem {Theorem}{Theorem}
\numberwithin{Theorem}{section}

\newtheorem {Lemma}[Theorem]    {Lemma}

\newtheorem {Proposition}[Theorem]{Proposition}
\newtheorem {Corollary}[Theorem]{Corollary}
\theoremstyle{definition}

\theoremstyle{remark}
\newtheorem{Remark}[Theorem]{Remark}
\newtheorem{Example}[Theorem]{Example}

\newcommand{\CS}{{\mathcal S}}
\newcommand{\LL}{{\mathcal L}}

\newcommand{\id}{{\mathit id}}

\newcommand{\charr}{{\mathit char}}

\newcommand{\tD}{\tilde{D}}

\newcommand{\Aa}{{\mathcal A}}

\newcommand{\CB}{{\mathcal B}}

\newcommand{\PP}{{\mathcal P}}

\newcommand{\Ss}{{\mathcal S}}

\def    \nat    {{\natural}}
\def    \F      {{\mathbb F}}

\def    \C      {{\mathbb C}}
\def    \R      {{\mathbb R}}

\def    \Z      {{\mathbb Z}}
\def    \N      {{\mathbb N}}
\def    \Q      {{\mathbb Q}}

\def    \12    {{\frac{1}{2}}}

\def    \p      {\partial}
\def    \codim  {\operatorname{codim}}

\def    \tr     {\operatorname{tr}}

\def    \GL     {\operatorname{GL}}

\def    \HF     {\operatorname{HF}}

\def    \H     {\operatorname{H}}

\def    \HM     {\operatorname{HM}}

\def    \CF     {\operatorname{CF}}

\def    \Fix     {\operatorname{Fix}}

\def    \bx     {\bar{x}}

\def    \bz     {\bar{z}}

\def    \MUCZ  {\operatorname{\mu_{\scriptscriptstyle{CZ}}}}

\def    \ind  {\operatorname{ind}}

\def \odd   {\scriptscriptstyle{odd}}
\def \even   {\scriptscriptstyle{even}}

\def    \str  {\operatorname{str}}
\def    \gr {\operatorname{gr}}
\def    \tr  {\operatorname{tr}}
\def    \sgn  {\operatorname{sign}}
\def    \coker {\operatorname{coker}}

\begin{document}

%%%%%%%%%%%%%%%%%%%%%%%%%%%%%%
%   TEXT FORMATTING

\setlength{\smallskipamount}{6pt}
\setlength{\medskipamount}{10pt}
\setlength{\bigskipamount}{16pt}

%%%%%%%%%%%%%%%%%%%%%%%%%%

%%%%%%%%%%%%%%%%%%%%%%%%%%

%%%%%%%%%%%           BEGINNING OF  TEXT

%%%%%%%%%%%%%%%%%%%%%%%%%%

\title[On the iterated Floer homology]{On the iterated Hamiltonian Floer homology}

\author[Erman \c C\. inel\. i]{Erman \c C\. inel\. i}

\author[Viktor L. Ginzburg]{Viktor L. Ginzburg}

\address{Department of Mathematics, UC Santa Cruz, Santa Cruz, CA
95064, USA} \email{scineli@ucsc.edu}

\address{Department of Mathematics, UC Santa Cruz, Santa Cruz, CA
95064, USA} \email{ginzburg@ucsc.edu}



\date{\today} 

\thanks{This work is partially supported by Simons
  Collaboration Grant 581382}

%\bigskip

\begin{abstract} 
  The focus of the paper is the behavior under iterations of the
  filtered and local Floer homology of a Hamiltonian on a
  symplectically aspherical manifold. The Floer homology of an
  iterated Hamiltonian comes with a natural cyclic group action. In
  the filtered case, we show that the supertrace of a generator of
  this action is equal to the Euler characteristic of the homology of
  the un-iterated Hamiltonian. For the local homology the supertrace
  is the Lefschetz index of the fixed point. We also prove an analog
  of the classical Smith inequality for the iterated local homology
  and the equivariant versions of these results.
\end{abstract}

\maketitle

%\bigskip

\tableofcontents

\section{Introduction}
\label{sec:intro+results}

\subsection{Introduction}
\label{sec:intro}
The main theme of this paper is the behavior of the filtered and local
Hamiltonian Floer homology under iterations. In particular, under
certain conditions we establish lower bounds on the rank of the Floer
homology of the iterated Hamiltonian via the homology of the original
Hamiltonian. The focus here is on the filtered or local rather than
global Floer homology.

Till very recently, virtually nothing was known about how the filtered
Floer homology behaves under iterations of a Hamiltonian $H$. With
local Floer homology the situation has been more encouraging. Namely,
consider an isolated one-periodic orbit $x$ of the Hamiltonian
diffeomorphism $\varphi_H$ generated by a Hamiltonian $H$.  As was
shown in \cite{GG:gap}, the local Floer homology group of the iterated
$k$-periodic orbit $x^k$ is independent of the order of iteration $k$,
up to a shift of degree, as long as $k$ is admissible, i.e., the
algebraic multiplicity of the eigenvalue $1$ is the same for
$D\varphi_H$ and $D\varphi^k_H$. Without the latter assumption, the
question is much more subtle and again essentially nothing was known
about the relation between the local Floer homology groups for $x$ and
$x^k$ prior to now.

The key ingredient in our arguments is the $\Z_k$-action on the
filtered or local Floer homology of $\varphi_H^k$. This action has a
canonical generator $g$ given, roughly speaking, by the
time-shift. The existence of such an action has been known for quite
some time.  For instance, in the framework of persistence modules this
action is used in \cite{PS,PSS,Zh} to study the placement of the
iterated Hamiltonians in the group of all Hamiltonian diffeomorphisms
in terms of the Hofer distance. In \cite{To} a slightly different
version of the action (less suitable for our purposes) is considered
and a supertrace relation, similar to the one proved in this paper, is
established for the Floer homology of commuting symplectomorphisms.

Our first result equates the supertrace of $g$ to the Euler
characteristic of the corresponding Floer homology group for
$\varphi_H$. (In the filtered case the manifold is assumed to be
symplectically aspherical.) For the local homology, this Euler
characteristic is simply the Lefschetz index of the fixed point.

The second result concerns a Floer theoretical version of the Smith
inequality. The classical version of this inequality relates the rank
of the total homology of a compact $\Z_p$-space $X$ with coefficients
in the field $\F=\F_p$ and the rank of the homology of the fixed point
set $X^{\Z_p}$, where $p$ is prime and $\F_p=\Z/p\Z$; see, e.g.,
\cite{Bo,Br} and references therein. Namely,
\begin{equation}
\label{eq:smith-intro}
\dim \H_*(X^{\Z_p};\F)\leq \dim \H_*(X; \F),
\end{equation}
where on the right and left hand sides we have the sum of dimensions
for all degrees.

From various perspectives, it is reasonable to expect a similar
inequality to hold for the filtered Floer homology of a Hamiltonian
diffeomorphism $\varphi_H$ of a symplectically aspherical closed
symplectic manifold when the role of the right hand side is taken by
the filtered Floer homology of $\varphi_H^p$ and the left hand side is
the filtered Floer homology of the original Hamiltonian diffeomorphism
$\varphi_H$. (For instance, $\Z_p$ acts on a basis of the Floer
complex of $\varphi_H^p$ and the elements of this basis, fixed up to a
sign by the action, form a basis for the Floer complex of
$\varphi_H$.)  For $p=2$, such an inequality is proved in \cite{Se} by
introducing and utilizing a $\Z_2$-equivariant version of the
pair-of-pants product.  Analogous results for all primes $p\geq 2$
have been recently obtained in \cite{ShZh}.

Our goal here is much less ambitious: we prove a version of the Smith
inequality for the local Floer homology $\HF_*(x)$ of an isolated
periodic orbit $x$. This, of course, also follows from the main
theorem in \cite{ShZh}, but the proof of the local version is much
simpler. Namely, we identify the local Floer homology with the generating
function homology (see \cite{Vi}) and then use the fact that the
latter is the homology of an actual $\Z_p$-space, where $p$ is the
iteration order. It remains to apply \eqref{eq:smith-intro} to this
space. The $\Z_p$-action on the Floer homology does not explicitly
enter the statement, but it is essential as a motivation and, somewhat
disguised, it also plays a central role in the argument.

The supertrace formula and the Smith inequality for local Floer
homology, although quite simple, have applications in dynamics of
degenerate Hamiltonian diffeomorphisms: these results, just as the
persistence theorem from \cite{GG:gap}, are used to guarantee that
$\HF_*(x^k)\neq 0$ under favorable conditions on $x$. We refer the
reader to \cite{GG:gap,Sh:HZ,Sh} for examples of such applications.

We also prove equivariant versions of these results pertaining to the
setting where the manifold is equipped with a symplectic action of a
finite group $G$ and $\varphi$ is $G$-equivariant. The two settings,
equivariant and iterated, are related via the so-called Dold's trick
discussed in Section \ref{sec:eq-str}. Most of the results of this
paper can be further generalized, at least partially, sometimes within
the same class of methods as used here. Some of these generalizations
are discussed in the introduction, but in most cases we have opted to
keep the framework, conditions and proofs simple and free of technical
complications. For instance, we left out the generalization of the
equivariant filtered supertrace formula to symplectomorphisms, which
requires additional requirements to have the action filtration
well-defined.

\subsection{Main results}
\label{sec:results} 

\subsubsection{Filtered supertrace formula}
Let $H \colon S^1 \times M \rightarrow \R$ be a one-periodic in time
Hamiltonian on a closed symplectically aspherical manifold
$(M, \omega)$. Here we identify $S^1$ with $\R/ \Z$. 
The filtered Floer homology of the iterated Hamiltonian
$H^{\nat k} \colon \R/k\Z \times M \rightarrow \R$,
  $H^{\nat k} (t, \cdot) := H(t, \cdot)$ (or the local Floer homology
of an iterated orbit $x^k$) carries a $\Z_k$-action, which comes with
a preferred generator
\[
  g_* \colon \HF_*^{(ka,\,kb)} (H^{\nat k}; \F) \rightarrow
  \HF_*^{(ka,\,kb)} (H^{\nat k}; \F).
\]
Here the coefficient ring $\F$ is a field and the endpoints $ka$ and
$kb$ are not in the action spectrum $\CS(H^{\nat k})$ of $H^{\nat
  k}$. Roughly speaking, $g$ arises from the time-shift map
$t \mapsto t+1$ applied to $k$-periodic orbits of $H^{\nat k}$. In
Section \ref{sec:def} we recall in detail the definition of the
action, paying particular attention to the role of orientations and
sign changes which is central to our results. The discussion of the
action in that section is essentially self-contained albeit brief.

Then, in Section \ref{sec:super}, we show that the supertrace
\begin{equation}
\label{eq:str}
  \str(g) :=
  \sum (-1)^i\tr
  \big( g_i \colon\HF_i^{(ka,\,kb)} (H^{\nat k}; \F)
  \rightarrow \HF_i^{(ka,\,kb)} (H^{\nat k}; \F)\big)
\end{equation}
of $g$ is equal to the Euler characteristic
$$
\chi\big(\HF_*^{(a,\,b)}(H; \F)\big)=\sum (-1)^i \dim \HF_i^{(a,\,b)} (H; \F)
$$
of $\HF_*^{(a,\,b)} (H; \F)$, viewed as an element of $\F$. In other
words, $\chi\big(\HF_*^{(a,\,b)}(H; \F)\big)$ is the image in $\F$ of
the Euler characteristic under the natural map $\Z\to\F$ or,
equivalently, the supertrace of the identity map on the homology of
$H$. For instance, $\chi\big(\HF_*^{(a,\,b)}(H; \F)\big)$ is precisely
equal to the Euler characteristic when $\F=\Q$ and is the Euler
characteristic modulo $p$ when $\F=\F_p$ and $p$ is prime.

The equality of the supertrace and the Euler characteristic is the
first result of this paper:

\begin{Theorem}[Supertrace Formula, I]
  \label{thm:1}
  Let $H$ be a one-periodic Hamiltonian on a closed symplectically
  aspherical manifold $(M, \omega)$.  Assume that $\F$ is a field.
  Then the supertrace of $g$ is equal to
  $\chi\big(\HF_*^{(a,\,b)}(H; \F)\big)$.
\end{Theorem}

Since both the supertrace and the Euler characteristic are additive with
respect to the action filtration, the proof of Theorem \ref{thm:1} reduces
to checking the orientation change at non-degenerate iterated orbits $x^k$
under the time-shift map $t \mapsto t+1$.

\begin{Remark}
  There seems to be no reason to believe that Theorem \ref{thm:1}
  would extend as stated to manifolds with $\omega|_{\pi_2(M)}\neq
  0$. The obstacle is that then there can be an orbit $x$ such that
  for no choice of capping its action is in $(a,\,b)$ but, for a suitable
  capping, $x^k$ has action in $(ka,\,kb)$.
\end{Remark}

For the local Floer homology $\HF_*(x^k)$ of an iterated orbit $x^k$,
Theorem \ref{thm:1} takes the following form.

\begin{Corollary}
\label{cor}
Let $M^{2n}$ be a symplectic manifold and $x \colon S^1 \rightarrow M$
be an isolated one-periodic orbit of a Hamiltonian $H$ on $M$.  Assume
that $\F$ is a field. For isolated iterates $x^k$ of $x$, the
supertrace of the generator is, up to the sign $(-1)^n$, equal to the
image in $\F$ of the Lefschetz index $\chi(x)$ of the fixed point $x(0)$.
\end{Corollary}

The corollary readily follows from Theorem \ref{thm:1} and its
proof. For the source of the sign $(-1)^n$ in
  Corollary \ref{cor} see the determinant identity \eqref{eq:cz-det}
  that relates the Conley-Zehnder and Lefschetz indices.  Recall also that when
  an isolated fixed point $x \in \Fix( \varphi)$ is degenerate, one can compute
  $\chi(x)$ as
 \[
 \chi(x)= \sum_{y \in \Fix(\tilde{\varphi}) \cap U}  
  \sgn \big( \det (D_y \tilde{\varphi} -I)\big),
 \]
 where $U$ is an isolating neighborhood of $x$ and $\tilde{\varphi}$
 is a non-degenerate perturbation of $\varphi$ supported in $U$. 

\begin{Remark}
  \label{rmk:top}
  One important consequence of Corollary \ref{cor} is that
  $\HF_*(x^k)\neq 0$ for all $k$ whenever $\chi(x) \neq 0$. Note in
  this connection that there exists a diffeomorphism $\varphi$ of
  $\R^n$, $n\geq 3$, with an isolated, for all iterations, fixed point
  $x$ such that $\chi(x)\neq 0$, but $\chi(x^k)=0$ for some $k$. (The
  construction of $\varphi$, described to us by Hern\'andez-Corbato,
  is non-obvious.) We do not know if this can also happen in the
  Hamiltonian setting, but Corollary~\ref{cor} does not rule out the
  existence of such Hamiltonian diffeomorphisms. Overall, the question
  if the local Floer homology of a homologically non-trivial orbit can
  vanish for some (isolated) iterates is still open; cf.\
  \cite{GG:gap}.
\end{Remark}

\begin{Remark}
  \label{rmk:non-contractible} Note that in Theorem \ref{thm:1}, the
  Floer homology groups of $\varphi_H$ and $\varphi_H^k$ ``count''
  only contractible periodic orbits and thus the theorem gives
  information only about such orbits. On the other hand, in Corollary
  \ref{cor} we do not need to require the orbit $x$ to be
  contractible. Indeed, recall that the composition of the flow of $H$
  near $x$ with a local loop of Hamiltonian diffeomorphisms does not
  change the local Floer homology, up to an even shift of grading; see
  \cite{Gi:CC,GG:gap}. (Here one should view the flow and the loop as
  defined on a neighborhood of the image of $x$ in the extended
  phase-space $M\times S^1$.) Then, composing the flow with a suitably
  chosen loop, we reduce the general case to the case where $x$ is a
  constant orbit.
\end{Remark}

\subsubsection{Equivariant supertrace formula}
\label{sec:eq-str}
Theorem \ref{thm:1} has an equivariant counterpart. To set the stage
for it, let us recall that the $k$-periodic orbits of $\varphi$ are in
one-to-one correspondence with the fixed points of the
symplectomorphism
$$
\varphi_k\colon (x_1,\ldots,x_k)\mapsto
\big(\varphi(x_2),\ldots,\varphi(x_k),\varphi(x_1)\big)
$$
of the product $M^k=M\times \ldots \times M$ ($k$-times). This
symplectomorphism commutes with the symplectic $G=\Z_k$-action
generated by the cyclic permutation
$$
g\colon (x_1,\ldots,x_k)\mapsto (x_2,\ldots,x_k,x_1).
$$
This observation from \cite{Do} allows one to translate problems about
periodic orbits to the equivariant context and is sometimes referred
to as Dold's trick; in the symplectic context it is utilized in
\cite{He}.  The fixed point set $(M^k)^G$ is the multi-diagonal and
the restriction $\varphi_k^G=\varphi_k\mid_{(M^k)^G} $ is just
$\varphi$. Under suitable additional conditions on $M$, the global
Floer homology $\HF_*(\varphi_k)$ is naturally isomorphic
(up to a shift of grading) to the Floer homology
$\HF_*(\varphi^k)$; \cite{LL}. (However, incorporating the action
filtration into this isomorphism requires some care; cf.\ Remark
\ref{rmk:thm:1e}.)

Using this construction as motivation, consider a symplectic action of
$G=\Z_k$ on a closed symplectically aspherical manifold $M$ and a
Hamiltonian diffeomorphism $\varphi\colon M\to M$ generated by a
$G$-invariant Hamiltonian $H$. (Thus all maps in the Hamiltonian
isotopy $\varphi_H^t$ are $G$-equivariant.) The fixed point set $M^G$
is a symplectic submanifold of $M$. Denote its connected components by
$C_j$, and by $H^G$ and, respectively, $H^G_j$ the restrictions of $H$
to $M^G$ and $C_j$, generating $\varphi^G:=\varphi\mid_{M^G}$ and
$\varphi^G_j:=\varphi\mid_{C_j}$. Then, similarly to the iterated
case, the action of $G$ on the filtered Floer homology
$\HF_*^{(a,\,b)} (H; \F)$ is defined. In particular, for a generator $g$
of $G$, we have the supertrace
$$
\str(g) :=
  \sum (-1)^i\tr
  \big( g_i \colon\HF_i^{(a,\,b)} (H; \F)
  \rightarrow \HF_i^{(a,\,b)} (H; \F)\big),
$$
where, as above, $\F$ is a field.

\begin{Theorem}[Supertrace Formula, II]
  \label{thm:1e}
  Under the above conditions, assume furthermore that every
  one-periodic orbit of $H^G$ which is contractible in $M$ is
  also contractible in $M^G$. Then
  \begin{equation}
    \label{eq:str2}
  \str(g)=\sum_j (-1)^{\codim C_j/2}\chi\big(\HF_*^{(a,\,b)} (H_j^G; \F)\big).
\end{equation}
\end{Theorem}
Here, again, the right hand side is interpreted as an element of
$\F$. The proof of this theorem is conceptually quite similar to the
proof of Theorem \ref{thm:1} and we only briefly outline it in Section
\ref{sec:pf-thm:1e}. As in the case of iterated maps, Theorem
\ref{thm:1e} yields a local result which we omit for the sake of
brevity.

\begin{Remark}
  \label{rmk:thm:1e}
  Some of the conditions of Theorem \ref{thm:1e} are imposed only for
  the sake of simplicity and most likely can be significantly relaxed,
  although at the expense of imposing more cumbersome
  requirements. For instance, it should be sufficient to assume that
  only $\varphi$ (rather than the entire flow $\varphi^t_H$) is
  $G$-equivariant.  However, then one would have to require the maps
  $\varphi^{G'}\colon M^{G'}\to M^{G'}$ to be Hamiltonian for all
  subgroups $G'$ of $G$, and also ensure that the actions for
  $\varphi$ and $\varphi^G_j$ are calculated in a consistent way, and
  a proof would most certainly call for a change of perspective.

Theorem \ref{thm:1e}, as stated, does not imply
    Theorem \ref{thm:1} via Dold's trick; for the map $\varphi_k$ is a
    symplectomorphism, but not a Hamiltonian diffeomorphism. There is
    global (i.e., unfiltered) version of Theorem \ref{thm:1e} for
    symplectomorphisms $\varphi$ commuting with $g$, equating (with
    our sign conventions) $(-1)^n\str(g)$ and the Lefschetz number
    $\chi\big(\varphi^G\big)$ of $\varphi^G$. (When $(a,\,b)=\R$ and
    $\varphi$ is Hamiltonian, $(-1)^n\chi\big(\varphi^G\big)$ is equal
    to the right hand side of \eqref{eq:str2}.)  This version is
    proved in \cite{To}, albeit with a somewhat
    different construction of the $G$-action on the Floer cohomology.
  This construction does not readily extend to the filtered
  setting. In fact, to incorporate the action filtration, one would
  certainly need to impose additional requirements on the
  symplectomorphism $\varphi$. While it is not clear what the optimal
  conditions are, it should be sufficient, for
    instance, to assume that $\varphi$ is isotopic, via an
  equivariant Hamiltonian isotopy, to a symplectomorphism with
  connected, none-empty fixed point set. This version of the theorem
  would imply Theorem \ref{thm:1} and require only minimal
  modifications to the proof, which we leave out; for it would be a
  long diversion from the main theme of this paper.
  \end{Remark}

\subsubsection{Comparison with topological equivariant index results}
It is interesting to compare the results from the previous section
with purely topological equivariant index results.

In its most general filtered version, Theorem \ref{thm:1} (or Theorem
\ref{thm:1e}), which is the main focus of this paper, does not have a
topological counterpart: there is simply no topological space with
homology naturally isomorphic to the filtered equivariant homology
$\HF_*^{(a,\,b)}(H; \F)$.  Nor is there, to the best of our knowledge,
a map whose (equivariant) index would be related to the supertrace
\eqref{eq:str}.

With the local version of this theorem the situation is
different. Namely, in the setting of Corollary \ref{cor}, one can
associate to $x^k$ the equivariant index of $\varphi^k$ at $x^k$ which
is an ellement of the Burnside ring of $\Z_k$; see, e.g.,
\cite{Cr,Fe,LR} and references therein. Consider the image of this
element in the virtual representation ring $R(\Z_k)$. Then, by the
equivariant Lefschetz formula, the supertrace of the generator of this
representation is simply the Lefschetz index of the fixed point
$x(0)$; see, e.g., \cite{LR}. On the other hand, by arguing exactly as
in the proof of Theorem \ref{thm:1}, one can check that this virtual
representation is, up to the sign $(-1)^n$, equal to
$\HF_{\even}(x^k)-\HF_{\odd}(x^k)$. Thus, the information
carried by Corollary \ref{cor} is equivalent to its topological
counterpart. However, proving this essentially requires reproving
Theorem \ref{thm:1} and the direct proof of the theorem given here is
arguably simpler than the proof of the equivariant Lefschetz formula.

The global (i.e., unfiltered) version of Theorem \ref{thm:1} for
Hamiltonian diffeomorphisms is void: the $\Z_k$-action on the global
Floer homology is trivial. To see this observe that
  for a $C^2$-small autonomous Hamiltonian (with a regular
  time-independent almost complex structure) the action is equal to
  the identity at the chain level. On the other hand, in the unfiltered
  setting, the actions on the Floer homology groups of any two
  Hamiltonians are isomorphic to each other via continuation. Hence
  the action is trivial on the global Floer homology for all
  Hamiltonians.

For symplectomorphisms, the question is more interesting. Namely,
under suitable additional conditions, one can expect an analog of
Theorem \ref{thm:1} to hold for symplectomorphisms. One way to
approach the question is to adapt our proof of Theorem \ref{thm:1} to
symplectomorphisms. Alternatively, one can use the symplectic version
of Dold's trick to reduce the problem to the global
symplectomorphism version of Theorem \ref{thm:1e} which follows from
the results in \cite{To}; cf.\ Remark \ref{rmk:thm:1e}. These two
methods, however, would use different definitions of the $\Z_k$-action
on the Floer homology. In any event, hypothetically, the information
obtained in this way in Theorems \ref{thm:1} and \ref{thm:1e} would be
equivalent to what is delivered by the equivariant (or iterated)
Lefschetz formula.

\subsubsection{Smith inequality in local Floer homology}
Corollary \ref{cor} implies that if the Lefschetz index of an orbit
$x$ (as a fixed point) is non-zero, then the local Floer homology of
its isolated iterates with field coefficients cannot be zero, cf.\
Remark \ref{rmk:top}. For the homology of prime iterates $x^p$ with
$\F_p=\Z/p\Z$-coefficients a stronger result holds.

\begin{Theorem}[Smith Inequality in Local Floer Homology, I]
\label{thm:2} 
Let $M$ be a symplectic manifold and $x$ be an isolated one-periodic
orbit of a Hamiltonian on $M$. Then for all isolated prime iterates $x^p$ of
$x$, we have the following inequality of total dimensions:
\begin{equation}
  \label{eq:smith}
\dim \HF_*(x; \F_p) \leq \dim \HF_*(x^p; \F_p).
\end{equation}
\end{Theorem}

This is a local result and, as in Remark
  \ref{rmk:non-contractible}, without loss of generality we may assume
that $M=\R^{2n}$ and $x=0$ is a constant one-periodic orbit of a germ
of $H$ at $x$. As has been mentioned above, this theorem is also a
consequence of the results of \cite{ShZh}.

Theorem \ref{thm:2} can be further generalized to the equivariant
setting as follows. Consider an action of a finite $p$-group $G$ on
$\R^{2n}$ fixing $x=0$, and let
$\varphi\colon (\R^{2n},0)\to (\R^{2n},0)$ be the germ of a
$G$-equivariant symplectomorphism. As above, denote by $(\R^{2n})^G$
the fixed point set of $G$. Then $\varphi$ restricts to a
symplectomorphism $\varphi^G$ of
$(\R^{2n})^G$. Furthermore, since a symplectic action of
  a compact group is symplectically linearizable near a fixed point (see,
  e.g., \cite[Thm.\ 22.2]{GS}), without loss of generality we can assume that
the $G$-action is linear and thus $(\R^{2n})^G$ is a linear subspace
of $\R^{2n}$. It is not hard to see that $\varphi$ is necessarily
Hamiltonian and that a generating Hamiltonian $H_t$, $t\in S^1$, can
be taken $G$-invariant for all $t$ and such that $dH_t(0)=0$. Changing
notation slightly, denote the local Floer homology of $\varphi$ at $x$
by $\HF_*(\varphi, x; \F_p)$.

\begin{Theorem}[Smith Inequality in Local Floer Homology, II]
\label{thm:2g} 
Assume that $x=0$ is an isolated fixed point of $\varphi$. Then
\begin{equation}
  \label{eq:smith2}
\dim \HF_*(\varphi^G, x; \F_p) \leq \dim \HF_*(\varphi, x; \F_p).
\end{equation}
\end{Theorem}

Theorem \ref{thm:2} can be proved directly by using generating
functions as it was done in the first version of this paper or
obtained from Theorem \ref{thm:2g} by applying Dold's trick; see
Section \ref{sec:eq-str}. We establish Theorem \ref{thm:2g} in Section
\ref{sec:smith}. The proof relies on the classical Smith theory. We
use generating functions to interpret $\HF_*(\varphi, x; \F_p)$ as the
singular homology of a pair of $G$-invariant subsets $(U, U_{-})$ of
$\R^{2n}$ and $ \HF_*(\varphi^G, x; \F_p)$ as the singular homology of
the fixed point set $(U^{G}, U_{-}^{G})$ of this action. Then Theorem
\ref{thm:2} follows from the relative version of the Smith inequality,
\eqref{eq:smith-intro}, applied to this pair:
\[
\dim \H_*(U^{G}, U_{-}^{G}; \F_p) \leq \dim \H_*(U, U_{-} ; \F_p).
\]
Finally note that when $G=\Z/\Z_2$ and $p=2$, a proof of Theorem
\ref{thm:2g} can be extracted from the results in \cite{SS}
establishing a variant of the Smith inequality for Lagrangian Floer
homology in the case of involutions.

\begin{Remark}
  \label{rmk:p->pl}
  In Theorem \ref{thm:2} the order of iteration is set to be equal
  exactly $p$ only for the sake of notational convenience. By Theorem
  \ref{thm:2g}, one can replace the iteration order $p$ by its power
  $p^\ell$, $\ell\in\N$, while still keeping the coefficient field
  $\F_p$. This, however, is also a formal consequence of Theorem
  \ref{thm:2}:
  $$
  \dim \HF_*(x; \F_p) \leq \dim \HF_*(x^p; \F_p) \leq \dim
  \HF_*(x^{p^2}; \F_p)\leq ...
  $$
  as long as $x^{p^\ell}$ are isolated, where in the second inequality
  we applied \eqref{eq:smith} to $x^p$, etc.
\end{Remark}

\medskip

\noindent{\bf Acknowledgements.} The authors are grateful to Daniel
Cristofaro-Gardiner, Ba\c sak G\"urel, Lois Hern\'andez-Corbato, Marco
Mazzucchelli, Egor Shelukhin and Jingyu Zhao for useful
discussions. The authors would also like to thank the referee for
calling their attention to the equivariant versions of these results.

\bigskip

\section{Preliminaries}
\label{sec:prelim}

\subsection{Conventions and basic definitions} 
\label{sec:notation}

Let, as above, $(M, \omega)$ be a closed symplectic manifold and let
$H \colon S^1\times M \rightarrow \R$ be a one-periodic in time
Hamiltonian on $M$. Here we identify $S^1$ with $\R/\Z$ and in what follows set $H_t :=H(t, \cdot)$. The
Hamiltonian vector field $X_H$ of $H$ is defined by
$i_{X_H}\omega= -dH$. The time-one map of the time-dependent flow of
$X_H$ is denoted by $\varphi_H$ and referred to as a \emph{Hamiltonian
  diffeomorphism}. In this paper we work with iterated Hamiltonians.
By $k$-th iteration $H^{\nat k}$ of $H$, we simply mean $H$ treated as
$k$-periodic. (We acknowledge the fact that strictly speaking it is
the map $\varphi_H$ rather than the Hamiltonian $H$ that is iterated.)
Time-dependent flows of $H$ and $H^{\nat k}$ agree and the time-$k$ map of
the latter is equal to $\varphi_H^k$.

A \emph{capping} of a contractible loop $x \colon S^1 \rightarrow M$
is a map $A \colon D^2 \rightarrow M$ such that $A\vert_{S^1} =x$. The
action of a Hamiltonian $H$ on a capped closed curve $\bx=(x,A)$ is
\[
\Aa_H(\bx)=-\int_A \omega + \int_{S^1} H(t,x(t)) \, dt.
\]
When $\omega\,\vert_{\pi_2(M)}=0$, the action $\Aa_H(\bx)$ is independent of
the capping. The critical points of $\Aa_H$ on the space of capped
closed curves are exactly the capped one-periodic orbits of $X_H$. The
set of critical values of $\Aa_H$ is called the \emph{action spectrum}
$\Ss (H)$ of $H$. These definitions extend to Hamiltonians of any
period in an obvious way. Note that the action functional is
homogeneous with respect to iteration, i.e., for iterated capped
orbits $\bx^k$ we have
\[
\Aa_{H^{\nat k}}(\bx^k)=k\Aa_H(\bx).
\]

A periodic orbit $x$ of $H$ is called \emph{non-degenerate} if the
linearized return map
$D\varphi_H \colon T_{x(0)}M \rightarrow T_{x(0)}M$ has no eigenvalues
equal to one. The \emph{Conley-Zehnder index} $\MUCZ (\bx) \in \Z$ of
a non-degenerate capped orbit $\bx$ is defined as in \cite{Sa,SZ}. The
index satisfies the determinant identity
\begin{equation}
\label{eq:cz-det}
\sgn\big( \det(D\varphi_H(x(0))-I)\big)= (-1)^{n-\MUCZ(\bx)},
\end{equation}
where $\dim M =2n$ (see Section 2.4 in \cite{Sa}), and is
independent of capping when $c_1(TM)\,\vert_{\pi_2(M)}=0$.

A symplectic manifold $(M, \omega)$ is called \emph{symplectically
  aspherical} if both $\omega$ and $c_1(TM)$ vanish on $\pi_2(M)$. In
this case we drop capping from the notation of the action and the
index.

\subsection{Floer homology and orientations} 
\label{sec:floer}  
In this section we recall the basics of filtered (and local) Floer
homology and orientations in Hamiltonian Floer homology. We refer the
reader to, e.g., \cite{FO, GG, HS, MS, Sa, SZ} for a detailed
treatment of Floer homology and to \cite{Ab, FH, Za} for a thorough
discussion of orientations.

\subsubsection{Filtered and local Floer homology} 
\label{sec:floer} Let, as above, $(M,\omega)$ be a closed
symplectically aspherical manifold and
$H \colon S^1\times M \rightarrow \R$ be a non-degenerate one-periodic
Hamiltonian on $(M, \omega)$. For a generic time-dependent almost
complex structure $J$, the pair $(H,J)$ satisfies the transversality
conditions; see \cite{FHS}.  Pick two points $a$ and $b \in \R$ not in
the action spectrum $\Ss(H)$ of $H$. The filtered Floer homology
$\HF_*^{(a,\,b)}(H)$ of $H$ is defined as the homology of the Floer
chain complex $\CF_*(H, J)$ restricted to the generators $x$ with
$\Aa_H(x) \in (a,\,b)$. (We omit the ground ring from the notation of
the homology when it is not essential.) If $H$ is degenerate, we take
a $C^2$-small non-degenerate perturbation $\tilde{H}$ of $H$ such that
$a$ and $b \notin \Ss(\tilde{H})$ and define $\HF_*^{(a,\,b)}(H)$ as
the filtered Floer homology $\HF_*^{(a,\,b)}(\tilde{H})$ of $\tilde{H}$.

The local Floer homology $\HF_*(H,x)$ (or just $\HF_*(x)$) of an
isolated one-periodic orbit $x \colon S^1 \rightarrow M$ is defined as
in \cite{Gi:CC,GG,GG:gap}. Recall that for a
  degenerate orbit $x$, similarly to the filtered case, $\HF_*(H, x)$ is
  defined as the homology of the Floer chain complex of a $C^2$-small
  non-degenerate perturbation $\tilde{H}$ of $H$, restricted to the
  generators contained in an isolating neighborhood of $x$. Note that
here we do not require $x$ to be contractible; see Remark
\ref{rmk:non-contractible}. The absolute grading of $\HF_*(H,x)$
depends on the choice of a trivialization of
$TM|_x$. In Section \ref{sec:def}, introducing the $\Z_k$-action on
filtered Floer homology, we imposed the condition that $(M,\omega)$ is
symplectically aspherical. In the case of local Floer homology this
assumption on $(M,\omega)$ is not needed. The reason is that after
composing $\varphi_H^t$ with a local loop of Hamiltonian
diffeomorphisms we can always assume that $x$ is a constant orbit of a
local Hamiltonian diffeomorphism of $(\R^{2n}, \omega_0)$; cf.\ Remark
\ref{rmk:non-contractible}.

\subsubsection{Orientations}
\label{sec:orient}
We describe the setup following \cite{Ab} and \cite{Za}. Let
$(M^{2n},\omega)$ be a closed symplectically aspherical manifold and
$H$ be a one-periodic in time Hamiltonian on $(M, \omega)$. For every
contractible one-periodic orbit $x$ of $H$, we choose a unitary
trivialization $\Psi$ of $TM$ along $x$ that comes from a capping of
$x$. We emphasize that, while choosing $\Psi$, one can work with any
complex structure $J$ on $TM$ along the capping.  Using $\Psi$, we
can linearize the Hamiltonian vector field along $x$ and write it as
$J_0S$, where $J_0$ is the multiplication by $i$ on $\C^n=\R^{2n}$ and
$S$ is a map from $S^1$ to the space of (symmetric) $2n\times 2n$ real
matrices. Set 
\begin{equation}
  \label{eq:D}
D=J_0\partial_t + S,
\end{equation}
where $t\in S^1$; see \cite[Sect.\ 2.2]{Sa}. This is a self-adjoint
first-order differential operator on $\R^{2n}$-valued functions on $S^1$.

Denote by $L^p\big(T^{(0,1)}\C\otimes\R^{2n}\big)$ the space of
$L^p$-sections of the bundle $T^{(0,1)}\C\otimes\R^{2n}$ over $\C$,
where $p>2$. Following \cite{Ab} (see also, e.g., \cite{Sc} for the
analysis details), consider an ``extension'' of $D$ to $\C$ given by the operator
\begin{align}
  \label{eq:tD}
\tilde{D}\colon W^{1,p}(\C, \R^{2n}) &\rightarrow L^p\big(T^{(0,1)}\C\otimes\R^{2n}\big)\nonumber
  \\
  \tilde{D}(X) &=\bar{\p}X +  \big( B\cdot X \big) \otimes d \bz,
%\tilde{D}(X) &=\partial_s X + J_0 \partial_tX + B\cdot X ,
\end{align}
where $B\colon \C\to \GL(2n)$ is asymptotic to $S$ in the following sense. Namely,
setting $z=e^{-s-2\pi it}$ and fixing the section $-d\bar{z}/\bar{z}$ of
$T^{(0,1)}\C$ over the cylinder $\R\times S^1=\{z\neq 0\}$, we can rewrite
$\tilde{D}$ as the operator
\begin{equation}
  \label{eq:tD2}
\tilde{D}_\infty(X) =\partial_s X + J_0 \partial_tX + B_\infty\cdot X
\end{equation}
on $\R^{2n}$-valued functions on the cylinder, where $B_\infty$ is a
function $\R\times S^1\to \GL(2n)$. In other words, 
  $$
  \tilde{D}=\tilde{D}_\infty\otimes\left(-\frac{d\bz}{\bz}\right)
  $$
over the cylinder $\{z\neq 0\}$. Then we require that
\[
B_\infty(e^{-s-2\pi i t})=S(t)
\] 
for $s \ll 0$. The behavior of $B$ outside a neighborhood of infinity
is inessential. Furthermore, the measure on $\C$ used in the
definition of the domain and the target of $\tD$ is required to be 
$ds\,dt$ near infinity. (It is worth pointing out that
$\tD_\infty$ defined by \eqref{eq:tD2} as a differential operator from
functions to functions extends to $\C$, but the extension is not
elliptic regardless of the choice of $B_\infty$: the leading term is
$-\bar{z}\p_{\bar{z}}$.
The relation between $B$ and $B_{\infty}$ is given by $B_\infty=-\bar{z}B$.)

Since $x$ is non-degenerate, $\tilde{D}$ is a Fredholm operator and
the index $\ind(\tD)$ of $\tilde{D}$ is equal to $n-\MUCZ (x)$;
see \cite{Sc}. Let
\[
\det(\tilde{D})=\det\big(\coker^*(\tilde{D}))\otimes \det(\ker(\tilde{D})\big)
\]
be the determinant line for the extended operator $\tilde{D}$. An
orientation choice for an orbit $x$ is an orientation choice for the
determinant line $\det(\tilde{D})$. (Here the determinant of a
finite-dimensional vector space is its top wedge power and
$\det(0)\cong\R$.)  The construction of $\tD$ depends on some
auxiliary data (the symplectic trivialization and the almost complex
structure along the capping, and $B$) forming a contractible space
$\CB$. Then the collection of determinant lines
$\det(\tilde{D})$ becomes a real line bundle $\LL\to\CB$: the
determinant bundle. This bundle is trivial since $\CB$ is contractible
and a choice of an orientation of one fiber is equivalent to a choice
of an orientation of~$\LL$.

Let us pick the complex structure in the construction of $\tD$, which
agrees along $x$ with an almost complex structure $J$ on $M$ such that
the pair $(H,J)$ is regular. Then a Floer trajectory connecting orbits
$x$ and $y$ gives rise, via gluing, to a canonical (up to
multiplication by a positive number) map between the determinant lines
$\det(\tilde{D}_{x})$ and $\det(\tilde{D}_{y})$, where the operators
$\tD_x$ and $\tD_y$ are associated with the orbits $x$ and $y$; see
\cite{Ab}. This map is then used to determine the signs in the Floer
differential by comparing the chosen orientations.

Next, let us consider the equivariant version of this construction,
which we will use in the proof of Theorem \ref{thm:1e}. Namely assume
that a finite group $G$ acts by symplectic transformations on $M$ and
that $H$ is $G$-invariant. Then every $g\in G$ sends one-periodic
orbits to one-periodic orbits and induces a homeomorphism
$\CB_x\to\CB_{g(x)}$ together with a natural lift
$\LL_x\to\LL_{g(x)}$. Depending on whether this lift is orientation
preserving or reversing we set $g\cdot x=\pm g(x)$ in the map
$\CF_*(H,J)\to \CF_*(H,g_*J)$ of the Floer complexes. Composing this
map with a continuation map from $(H, J)$ to $(H,g_*J)$ constant in
$H$, we obtain a $G$-action on the filtered Floer homology of $H$. In
what follows, we will need to see more explicitly how the sign, which
we denote by $\sgn(g)$, is calculated when $g(x)=x$, i.e., $x$ is an
orbit invariant under the $G$-action.  It turns out that $\sgn(g)$ is
determined by a certain component of the flow along $x$; see Lemma
\ref{lemma:index2}.

Note that to calculate this sign we can use any, not necessarily
regular, almost complex structure along the capping of $x$. For
instance, if the capping is $G$-invariant we can work with a
$G$-invariant auxiliary data. In this case the calculation can be best cast
in the following somewhat formal framework. Let $G$ be a finite group
acting by symplectic linear transformations on $\R^{2m}$ and let
$\Phi_t$ be a linear $G$-equivariant time-dependent flow. We set
$$
S=-J_0\dot{\Phi}_t\Phi_t^{-1}
$$
to be the matrix of a quadratic $G$-invariant, one-periodic in time
Hamiltonian on $\R^{2m}$ generating $\Phi_t$. The only requirement we
impose is that the time-one map $\Phi_1$ is non-degenerate, i.e., $\det (\Phi_1-I) \neq 0$.

Consider the operator $D$ associated with this data as in \eqref{eq:D}
and denote by $\tD$ its extension defined by \eqref{eq:tD}, where $B$
is $G$-invariant. Then $\tD$ is again Fredholm with index
$\ind(\tD)=m-\MUCZ(\Phi)$ and, by construction, $\tD$ is
equivariant. As a consequence, $G$ acts on $\det(\tD)$ and we obtain a
well-defined homomorphism $\sgn\colon G\to\Z_2=\{1,\,-1\}$ given by
the sign of this action.

Three standard remarks are due. The first one is that this
homomorphism is invariant under an equivariant homotopy of $\Phi_t$
(or of $S$) as long the time-one map $\Phi_1$ remains
non-degenerate. Secondly, when $S$ and the underlying
symplectic vector space decompose as a direct sum, $\det(\tD)$
decomposes as a tensor product, and the $\sgn$ homomorphism is the
product of the sign homomorphisms for the individual terms. In other
words, the construction is multiplicative with respect to direct
sum. Finally, if we replace the finite group $G$ by a compact Lie group,
the construction still goes through and the $\sgn$ homomorphism is
continuous. In particular, it is trivial, i.e., identically equal to
$1$, when the group is connected.

\begin{Example}
  \label{ex:circle}
  Assume that $m=1$, the flow $\Phi_t$ is the rotation by $t\theta$ and
  $G$ is a cyclic group whose generator $g$ acts by a rational rotation of
  $\R^2$ by some angle possibly different from $\theta$. Then
  $\sgn(g)=1$. Indeed, in this case we can extend the $G$-action to an
  $S^1$-action and apply the previous remark.
\end{Example}

We will need the following

\begin{Lemma}
  \label{lemma:index}
  Assume that $G$ is cyclic with generator $g$ and that the fixed
  point set $(\R^{2m})^G$ is zero, i.e., $g$ has no eigenvectors
  with eigenvalue $1$. Then
  \begin{equation}
    \label{eq:index}
  \sgn(g)=(-1)^{m-\MUCZ(\Phi)}=(-1) ^{\ind(\tD)}.
\end{equation}
\end{Lemma}

\begin{proof}
Note that we only need to prove that $\sgn(g)$ is
    equal to either one of the two terms on the right. We will focus
    on the second one: $\sgn(g)=(-1) ^{\ind(\tD)}$. Then
    \eqref{eq:index} is a consequence of the general fact that
    whenever a cyclic group $G$ acts on a finite-dimensional real
    vector space $V$ with $V^G=\{0\}$, a generator $g$ of $G$ acts by
    multiplication by $(-1)^{\dim V}$ on the determinant line
    $L=\det(V)$. Indeed, $V$ decomposes as a direct sum of two
    invariant subspaces $V'$ and $V''$ such that, on $V'$, the
    generator $g$ acts as multiplication by $-1$ and $V''$ is a direct
    sum of other irreducible representations of $G$. In other words,
    $V''$ is a direct sum of two-dimensional subspaces on which $g$
    acts by rotation in an angle other than $\pi$. Thus the action of
    $g$ on $V''$ is orientation preserving and $\dim V''$ is even. It
    follows that $g$ acts on $L=\det(V')\otimes\det(V'')$ 
    by multiplication by $(-1)^{\dim V'}=(-1)^{\dim V}$.

 Applying this observation to $V=\ker\tD$ and
    $V=\coker \tD$ and using the fact that by definition
  $$
  \ind(\tD)=\dim\ker(\tD)-\dim\coker(\tD)
  $$
  we obtain \eqref{eq:index}.
  \end{proof}

  Now we are in a position to return to the question of calculating
  the sign for an orbit $x$ fixed by the action. Thus, let $g$ be a
  generator of a cyclic group $G$ acting by symplectic tranformations
  on $M^{2n}$ and let $H$ be a $G$-invariant Hamiltonian. Then the
  flow $\varphi_H^t$ restricts to the Hamiltonian flow on $M^G$
  generated by $H^G:=H\mid_{M^G}$.  Consider a one-periodic orbit $x$
  of $H$ fixed by the $G$-action. Then $x$ necessarily lies in $M^G$
  and when $x$ is viewed as a periodic orbit of $H^G$ we denote it by
  $x^G$.  Recall also that by definition $g\cdot x=\sgn(g) x$ in the
  map of the Floer complexes.

  \begin{Lemma}
    \label{lemma:index2}
    Assume that $x^G$ is contractible in $M^G$. Then
    \begin{equation}
      \label{eq:index2}
    \sgn(g)=(-1)^{m-\MUCZ(x)+\MUCZ(x^G)},
  \end{equation}
  where $2m$ is the codimension of the connected component of $M^G$
  containing $x$.
  \end{Lemma}

  \begin{proof}Fix a capping $A\colon D^2\to M$ of
      $x^G$ in $M^G$. Both the pull-back $A^*TM^G$ of the tangent
      bundle $TM^G$ and the pull-back $A^*(TM^G)^\omega$ of the
      symplectic orthogonal bundle to $M^G$ are trivial over
      $D^2$. Moreover, since actions of compact groups are rigid (see,
      e.g., \cite[Appendix B]{GGK}), $A^*(TM^G)^\omega$ is
      equivariantly trivial, i.e., it is equivariantly isomorphic to
      the product $D^2\times V$, where $V$ is a real symplectic
      representation of $G$ such that $V^G=\{0\}$. Let us fix a
      trivialization of $TM^G$ and an equivariant trivialization of
      the normal bundle over the capping.

   The flow $D\varphi_H^t$ along $x$ respects the
      direct sum decomposition $T_{M^G}M=TM^G\oplus (TM^G)^\omega$.
      Let $\Phi_t$ be the normal component of $D\varphi_H^t|_x$,
      viewed as a family of symplectic maps $V\to V$ commuting with
      the $G$-action. Operator $D_x$ splits as a direct sum of two
      operators: one for the tangent component
      $D\varphi_{H^G}^t\mid_{x^G}$ of the flow and one for the normal
      component $\Phi_t$. Hence, one can also chose a split extension
      $\tilde{D}_x$ of $D_x$, and $\sgn(g)$ is the product of the
      signs of $g$ for these two extensions. The $G$-action on the
      part tangent to $M^G$ is trivial, and hence the sign coming from
      that component is $+1$. Therefore, $\sgn(g)$ is equal to the
      sign coming from the normal component, which is determined by
      \eqref{eq:index}.

    Since the Conley--Zehnder index is additive, we have
    $\MUCZ(\Phi)=\MUCZ(x)-\MUCZ\big(x^G\big)$, and the lemma follows
    from \eqref{eq:index}.
  \end{proof}

\section{$\Z_k$-action on the iterated Floer homology}
\label{action}

\subsection{Definition of the action}
\label{sec:def}
In this section we recall the definition of the $\Z_k$-action on the
$k$-iterated Floer homology, treating in detail the role of
orientations and sign changes -- a point of particular importance to
us here. The construction is quite similar to the $G$-action on the
Floer homology in the equivariant setting discussed in Section
\ref{sec:orient}. (In fact, one can use Dold's trick to reduce the
iterated case to the equivariant one, but the resulting description of
the action is more involved and relies on the Floer homology of
symplectomorphisms rather than Hamiltonian diffeomorphisms making the
action filtration more difficult to work with.)

As before, let $(M,\omega)$ be a closed symplectically aspherical
manifold and $H \colon S^1\times M \rightarrow \R$ be a one-periodic
in time Hamiltonian on $(M, \omega)$. Denote by $H^{\nat k}$ the $k$th
iteration of $H$. Our goal is to define, for $a$ and $b$ not in
$\CS(H^{\nat k})$, a generator
\[
g \colon \HF_*^{(a,\,b)} (H^{\nat k}) \rightarrow  \HF_*^{(a,\,b)} (H^{\nat k})
\]
of the action and show that $g^k=\id$. The definition works for any
coefficient ring, and it extends word-for-word to the local Floer
homology $\HF_*(x^k)$ of an isolated iterated orbit $x^k$.

Fix such $a$ and $b$. If the iterated Hamiltonian $H^{\nat k}$ is
degenerate, we perturb $H$ so that the $k$th iteration of the
perturbed Hamiltonian $\tilde{H}^{\nat k}$ is non-degenerate and
$C^2$-close to $H^{\nat k}$. Let $J$ be a $k$-periodic almost complex
structure, which is one-periodic in a tubular neighborhood of every
$k$-periodic orbit of $\tilde{H}^{\nat k}$. For a generic almost
complex structure $J$ in this class, the pair
$(\tilde{H}^{\nat k}, J)$ satisfies the transversality conditions (see
\cite[Rmk.\ 5.2]{FHS}). On the chain level, the generator $g$ is the
composition $CR$ of the time-shift map
\begin{align*}
  R \colon \CF_*^{(a,\,b)}\big(\tilde{H}^{\nat k}_t, J_t\big)
  &\rightarrow \CF_*^{(a,\,b)}\big(\tilde{H}^{\nat k}_{t+1}, J_{t+1}\big) 
  \\
  x(t) &\mapsto \pm \,x(t+1)
\end{align*}
with a continuation map
\[
  C \colon \CF_*^{(a,\,b)}\big(\tilde{H}^{\nat k}_{t+1}, J_{t+1}\big)
  \rightarrow  \CF_*^{(a,\,b)}\big(\tilde{H}^{\nat k}_t, J_t\big)
\]
which is constant in the Hamiltonian:
$\tilde{H}^{\nat k}_{t+1}=\tilde{H}^{\nat k}_t$. (Here and throughout
we use the same notation $g$ for the chain map and the induced map in
homology.) Note that the domain and the target complexes are identical
as graded vector spaces, but in general not as complexes. Furthermore,
these Floer complexes come with a preferred basis and $R$ sends
generators to generators, up to a sign.  When $\F=\Z_2$, the
time-shift map $R$ is clearly a chain map since all the data,
including Floer trajectories, is just shifted in time. If $J$ is
one-periodic in time, the domain and the target complexes agree and we
have $C=\id$. Then $g=R$ is obviously $k$-periodic already on the
level of complexes. This need not be the case in general.

Below we will discuss how the signs in the time-shift map are
determined and then show in Proposition \ref{prop:action} that $CR$
induces a $k$-periodic map, i.e., a $\Z_k$-action, on the level of
homology; cf.\ \cite[Lemma 2.6]{To}. (A proof of the proposition can
also be found in, e.g., \cite{PS,Zh}, but there the signs are
implicit.)

Let us first explain how to fix the orientations for periodic orbits
generating the domain $\CF_*^{(a,\,b)}(\tilde{H}^{\nat k}_t, J_t)$ of
$g$. For the orbits that are time-shifts of each other, e.g. $x(t)$
and $x(t+1)$, using the same capping we choose a one-periodic trivialization of $J$ along the orbits. Morse precisely, we take a one-periodic $\tilde{J}$ on the capping that is equal to $J$ along the orbit and then choose a one-periodic trivialization for $\tilde{J}$. Since all the
data is one-periodic in tubular neighborhoods of the orbits, the
asymptotic operator for $x(t+1)$ becomes the time-shift of the
asymptotic operator $D_x$ for $x(t)$. Hence we may use the time-shift
of the extended operator (see Section \ref{sec:orient}) for $x(t)$ to
determine the orientation line for $x(t+1)$. We choose extended
operators for un-iterated $k$-periodic orbits according to the rule
above. If a periodic orbit is iterated, we require its extended
operator to have the same period as the orbit.  In other words, we
extend asymptotic operators by preserving their minimal period. Once
the extended operators are fixed, we choose any orientation of their
determinant lines.

In the target $\CF_*^{(a,\,b)}(\tilde{H}^{\nat k}_{t+1}, J_{t+1})$ we
use the same extended operators and we make the same orientation
choices as above. Now we have a natural time-shift map 
\[
R_* \colon \det( \tilde{D}_{x(t)}) \rightarrow \det( \tilde{D}_{x(t+1)})
\]
between the
determinant lines for $x(t)$ and $x(t+1)$. Using the map $R_*$ we compare
the chosen orientations. The sign of $R$ is positive if $R_*$ is orientation preserving and
negative otherwise. Next we show that  the induced map $g$ in
homology generates a $\Z_k$-action.

\begin{Remark} In our choice of the orientation data, $k$-iterated
  orbits in the domain and the target of $R$ occur with the same
  orientation. We could have chosen the orientations for periodic
  orbits in the domain and the target so that all signs are positive
  by allowing this orientations to differ. However, in this case,
  a sign change would occur in the continuation part $C$. For
  instance, this would be the case for iterated orbits when in our
  setting the negative sign occurs in~$R$.
\end{Remark}

\begin{Proposition}
  \label{prop:action}
  The map $g$ generates a $\Z_k$-action in the Floer homology. In
  other words, $g^k = \id$.
\end{Proposition}

\begin{proof} We will show that 
\[
  (CR)^k\colon \CF_*^{(a,\,b)}\big(\tilde{H}^{\nat k}_{t}, J_{t}\big)
  \rightarrow \CF_*^{(a,\,b)}\big(\tilde{H}^{\nat k}_{t}, J_{t}\big) 
\]
can be written as a composition of continuation maps, and hence is the
identity on the level of homology. Denote by $R_i$, $i\in\Z$, the
time-shift map
\[
  R_i \colon \CF_*^{(a,\,b)}\big(\tilde{H}^{\nat k}_{t+i}, J_{t+i}\big)
  \rightarrow \CF_*^{(a,\,b)}\big(\tilde{H}^{\nat k}_{t+(i+1)}, J_{t+(i+1)}\big) 
\]
and by $C_{i+1}$ the continuation map
\[
C_{i+1} \colon \CF_*^{(a,\,b)}\big(\tilde{H}^{\nat k}_{t+(i+1)},
J_{t+(i+1)}\big)
\rightarrow  \CF_*^{(a,\,b)}\big(\tilde{H}^{\nat k}_{t+i}, J_{t+i}\big)
\]
given by shifting in time the continuation data in $C$ (i.e., a homotopy
in $H$ and $J$). With this new notation $C=C_1$ and $R=R_0$. Note that
all $R_i$'s are essentially the same map and $C_{i+1}$ is the
time-shift of $C_i$, up to signs determined by the relation
\[
\big\langle x, y \big\rangle_{C_i}=\big\langle R_i(x), R_{i-1}(y) \big\rangle_{C_{i+1}}
\]
where $\big\langle x, y \big\rangle_{C_i}$ denotes the coefficient of
$y$ in the image $C_i(x)$ of $x$. In other words the identity
\begin{equation}
\label{eq:commutes}
R_{i-1}C_i= C_{i+1}R_i
\end{equation}
holds. By applying (\ref{eq:commutes}) to $(C_1R_0)^k$ we conclude that
\[
(C_1R_0)^k=(C_1C_2\cdots C_k )( R_{k-1} \cdots R_1R_0)= C_1C_2\cdots C_k,
\]
which induces the identity map on the level of homology.
\end{proof}

\begin{Remark}
  The requirement that $J$ is one-periodic near the $k$-periodic
  orbits of $H$ is not essential for the definition of the
  $\Z_k$-action -- it can be avoided by considering the entire
  determinant bundle rather than specific determinant lines; see
  Section \ref{sec:orient}. However, it becomes useful in the proof of
  Proposition \ref{prop:action} and also in the proof of Theorem
  \ref{thm:1} below.
\end{Remark}  

\subsection{Calculation of the supertrace -- Proof of Theorem \ref{thm:1}}
\label{sec:super}
Now we are in a position to show that the supertrace of the map
\[
g \colon \HF_*^{(ka,\,kb)} (H^{\nat k}; \F) \rightarrow  \HF_*^{(ka,\,kb)} (H^{\nat k}; \F)
\]
is equal to the Euler characteristic of $\HF_*^{(a,\,b)}(H; \F)$ and
thus prove Theorem \ref{thm:1}. Here the coefficient ring $\F$ is
required to be a field. This allows us to split the chain complex as a
direct sum of kernels and images and conclude that the supertrace at
the chain level is equal to the supertrace on the level of homology.
We reduce the problem to the case where $H^{\nat k}$ is non-degenerate
and the interval $(ka, kb)$ contains a single action value. These are
non-restrictive assumptions since one can compute the supertrace
$\str(g)$ at the chain level (\cite[Theorem
  23.2]{Br2}) and both $\str(g)$ and the Euler characteristic are
additive under direct sum, and hence additive with
respect to the action filtration.

We further simplify the problem by choosing a continuation 
\[
  C \colon \CF_*^{(ka,\,kb)}\big(H^{\nat k}_{t+1}, J_{t+1}\big)
  \rightarrow  \CF_*^{(ka,\,kb)}\big(H^{\nat k}_t, J_t\big)
\]
with a homotopy constant in the almost complex structure on a tubular
neighborhood of every $k$-periodic orbit. For a generic homotopy in
this class the transversality conditions are satisfied (see
\cite[Rmk.\ 5.2]{FHS}). As a result, since $(ka,\,kb)$ contains a
single action value, the continuation part of $g$ becomes the identity. We
have
\[
\str(g)=\str(CR)=\str(R)=\sum_{\substack{ x \in \PP(\tilde{H}) \\ 
\mathcal{A}_{\tilde{H}}(x) \in (a,\,b)}} \pm \,(-1)^{\MUCZ(x^k)},
\]
where $\pm$'s are the signs in the time-shift map $R$ (see Section
\ref{sec:def}). Note that only the iterated one-periodic orbits appear
in the trace formula. Next we show that the sign of a $k$-iterated
orbit $x^k$ is equal to $(-1)^{\MUCZ(x)-\MUCZ(x^k)}$ (see \cite[Thm.\
3]{BM}), and hence conclude that
\[
\str(g)=\sum_{\substack{ x \in \PP(\tilde{H}) \\ 
    \mathcal{A}_{\tilde{H}}(x) \in (a,\,b)}}  (-1)^{\MUCZ(x)} =
\chi \big(\HF_*^{(a,\,b)}(H; \F) \big) .
\]

\begin{Proposition}
  \label{prop:it-sign}
The sign for a $k$-iterated orbit $x^k$ is equal to $(-1)^{\MUCZ(x)-\MUCZ(x^k)}$.
\end{Proposition}

\begin{proof}
  Recall that the extended operator $\tilde{D}_{x^k}$ for a
  $k$-iterated orbit $x^k$ is one-periodic; see Section
  \ref{sec:def}. The time-shift map generates a $\Z_k$-action on the
  kernel an cokernel of $\tilde{D}_{x^k}$. Irreducible representations
  of this action are rotations, and the multiplications by one and
  negative one. We are interested in the parity of the dimension of
  $(-1)$-eigenspace.

  Observe that $(+1)$-eigenspace corresponds to the kernel and the
  cokernel of the extended operator $\tilde{D}_x$ for the un-iterated
  orbit $x$. So the parity of the dimension of the $(-1)$-eigenspace
  is equal to the parity of the difference between the Fredholm
  indices of $\tilde{D}_{x^k}$ and $\tilde{D}_{x}$, which in turn is
  equal to $\MUCZ(x)-\MUCZ(x^k)$.
\end{proof}

\begin{Remark}
  Alternatively, one can adapt the proof of Lemma \ref{lemma:index} to
  establish Proposition \ref{prop:it-sign}
\end{Remark} 

We have the same supertrace formula for the generator $g$ of the
$\Z_k$-action on the local Floer homology $\HF_*(x^k; \F)$ of an
isolated iteration $x^k$ of a one-periodic orbit
$x \colon S^1 \rightarrow M$.  In this case, using \eqref{eq:cz-det},
one can also write the formula as
\[
\str\big(g_* \colon \HF_*(x^k; \F) \rightarrow \HF_*(x^k; \F) \big)=(-1)^n\chi(x),
\]
where $\chi(x)$ is the Lefschetz index of the fixed point
$x(0)$ and $\dim M =2n$.

\begin{Remark}
  One consequence of the proof of Theorem \ref{thm:1} is that
  \begin{equation}
    \label{eq:virt}
 \CF_{\even}^{(ka,\,kb)}\big(H^{\nat k}, J\big)-\CF_{\odd}^{(ka,\,kb)}\big(H^{\nat k},
 J\big)=
 \HF_{\even}^{(ka,\,kb)}\big(H^{\nat k}\big)-\HF_{\odd}^{(ka,\,kb)}\big(H^{\nat k}\big)
  \end{equation}
  in the virtual representation ring $R(\Z_k)$ when
  $\charr\,\F=0$. This is not immediately obvious because while $\Z_k$
  acts on the graded vector space
  $\CF_{*}^{(ka,\,kb)}\big(H^{\nat k}, J\big)$ this action need not
  commute with the Floer differential; \eqref{eq:virt} readily
  follows from the fact that both sides have the same character
  $\chi(g)=\str(g)$ for all $g\in\Z_k$. This also shows that the right
  hand side of \eqref{eq:virt} lies in the subgroup $B(\Z_k)$ of
  $R(\Z_k)$ generated by permutation representations. For $\F=\Q$ this
  statement is void, for then $B(\Z_k)=R(\Z_k)$, but for other
  fields (e.g., $\F=\C$) $B(\Z_k)$ is a proper subgroup. Furthermore,
  as a result, all generators of the $\Z_k$-action on the filtered
  Floer homology have the same supertrace.  (In general, for a
  $\Z_k$-action on a (graded) vector space, the (super)trace of a
  generator depends on the generator.)  These observations carry over
  to the setting of Theorem \ref{thm:1e}.
\end{Remark}

\subsection{Proof of Theorem \ref{thm:1e}}
\label{sec:pf-thm:1e}
Let $G$ be a finite group acting on a closed symplectically aspherical
manifold $M$ by symplectomorphisms and let $H$ be a $G$-invariant (for
all times) Hamiltonian on $M$. Then $H$ generates a $G$-equivariant
time-dependent Hamiltonian flow $\varphi_H^t$ and we have a $G$-action on
the filtered Floer homology of $H$; see Section \ref{sec:orient}. 

Our goal is to prove Theorem \ref{thm:1e}. Throughout the proof we
will assume that, as in the theorem, $G$ is cyclic and denote by $g$ a
generator of $G$. The argument, based on Lemmas \ref{lemma:index} and
\ref{lemma:index2}, is very close to the proof of Theorem \ref{thm:1}
and we only outline it. It is not hard to see that by applying an
equivariant perturbation to $H$ we may assume that $\varphi$ and hence
$\varphi^G:=\varphi\mid_{M^G}$ are non-degenerate.

Let $x$ be a one-periodic orbit of $H$ fixed by the $G$-action. Thus
$x$ lies on $M^G$ and can also be viewed as a one-periodic orbit of
$H^G$. Then $g x=\pm x$ in the Floer complex and the $\pm$ sign, which
we denote by $\sgn(g)$ as in Lemma \ref{lemma:index2}, is determined
by the action of $g$ on the determinant line bundle of $x$; see
Section \ref{sec:orient}. Denote by $2m$ the codimension of the
connected component of $M^G$ containing $x$. To prove the theorem, it
suffices to show that
$$
\sgn(g)(-1)^{\MUCZ(x)}=(-1)^{m}(-1)^{\MUCZ(x^G)}
$$
or equivalently
$$
\sgn(g)=(-1)^{m}(-1)^{\MUCZ(x)-\MUCZ(x^G)},
$$
which immediately follows from \eqref{eq:index2}.

\section{Smith theory -- Proofs of Theorems \ref{thm:2} and \ref{thm:2g}}
\label{sec:smith}
In this section we prove Theorem \ref{thm:2} and its
equivariant counterpart Theorem~\ref{thm:2g}.

\begin{proof}[Proof of Theorem \ref{thm:2g}]
  As in the statement of the theorem, let $G$ be a finite $p$-group acting on
  $(\R^{2n}, \omega_0)$ by linear symplectomorphisms and $\varphi$ be a
  $G$-equivariant local Hamiltonian diffeomorphism of
  $(\R^{2n}, \omega_0)$ with an isolated fixed point at the origin. We
  denote by $\varphi^G$ the restriction of $\varphi$ to the fixed point set
  $(\R^{2n})^G \subset \R^{2n}$ of $G$. Note that $(\R^{2n})^G$ is a
  symplectic linear subspace of $(\R^{2n}, \omega_0)$ and $\varphi^G$ is a local
  Hamiltonian diffeomorphism of $\big((\R^{2n})^G, \omega_0\vert_{(\R^{2n})^G}\big)$
  with isolated fixed point at the origin. We need to show
  that the local Floer homology of $\varphi$ and $\varphi^G$ with
  $\F_p$-coefficients satisfy the Smith inequality
\[
\dim \HF_*(\varphi^G, x; \F_p) \leq  \dim \HF_*(\varphi, x; \F_p).
\]
We do this by interpreting $\HF_*(\varphi, x; \F_p)$ as the singular
homology of a pair of sufficiently nice topological spaces
$(U, U_{-})$ which carry a $G$-action; and $\HF_*(\varphi^G, x; \F_p)$ as
the singular homology of the fixed point set of this action. Then the
result is a consequence of \eqref{eq:smith-intro}, the classical Smith
inequality, applied to this pair:
\begin{equation}
  \label{eq:Smith-top}
\dim \H_*(U^{G}, U_{-}^{G}; \F_p) \leq \dim \H_*(U, U_{-} ; \F_p);
\end{equation}
see, e.g., \cite[Section 3, Thm.\ 7.9]{Br} or \cite{Bo}. Note that
\eqref{eq:Smith-top} for a general finite $p$-group $G$ follows from the case
$G=\Z_p$, since a finite $p$-group is nilpotent and for a normal subgroup
$H \leq G$:
\[
(U^{H}, U_{-}^H)^{G/H}=(U^G, U_{-}^{G}).
\]

Fix a complement $W$ to $V:= (\R^{2n})^G \subset \R^{2n}$ and identify
$\R^{2n}$ with $V\times W$.  The $G$-action on $(\R^{2n}, \omega_0)$
extends to the twisted product
$$
\bar{\R}^{2n} \times \R^{2n}=(\R^{2n}\times \R^{2n}, -\omega_0 \times
\omega_0)
$$
by acting coordinate-wise. Let
$\triangle \subset \bar{\R}^{2n} \times \R^{2n}$ denote the diagonal
and $\gr(D\varphi)\subset \bar{\R}^{2n} \times \R^{2n}$ be the graph
of the linearization $D\varphi$ of $\varphi$ at the origin. Note that
both $\triangle$ and $\gr(D\varphi)$ are $G$-invariant subsets of
$ \bar{\R}^{2n} \times \R^{2n}$.  Below we show that there is a
$G$-invariant common Lagrangian complement $N$ to $\triangle$ and
$\gr(D\varphi)$.

First assume that $\triangle$ and $\gr(D\varphi)$ are transverse
and choose a basis $\{e_i\}$, $\{f_i\}$ such that
$\omega_0 (e_i, f_j) = \delta_{ij}$. Since $\triangle$ and
$\gr(D\varphi)$ are $G$-invariant; the span of $\{e_i+f_i \}$ gives a
$G$-invariant Lagrangian $N$, which is complement to both $\triangle$
and $\gr(D\varphi)$. Next we consider the case
$\triangle=\gr(D\varphi)$. Choose a $G$-invariant metric on
$\bar{\R}^{2n} \times \R^{2n}$. Together with $\omega_0$, the metric
defines a compatible $G$-invariant complex structure $J$ on
$\bar{\R}^{2n} \times \R^{2n}$ and, using $J$, we set $N=J\triangle$. 

The general case: The intersection $S:=\triangle \cap \gr(D\varphi)$
is $G$-invariant, so we can split $\triangle = S \bigoplus S^{'}$
where $S^{'}\subset \triangle$ is a $G$-invariant subspace (and
similarly $\gr(D\varphi)$) and reduce the problem to the two cases
above.

In what follows we identify the twisted product
$ \bar{V} \times V =(V\times V, -\omega_0\vert_V \times
\omega_0\vert_V) $ with its image (given by inclusion) in
$\bar{\R}^{2n} \times \R^{2n}$.  Under this identification the
diagonal in $\bar{V} \times V$ is mapped to $\triangle^G$ and the
graph $\gr(D\varphi^G)$ of the linearization of $\varphi^G$ at the
origin is mapped to $\gr(D\varphi)^G$. Observe that the fixed point
set $N^G \subset N$ of the $G$-invariant complement $N$ found above is
a common Lagrangian complement to the diagonal $\triangle^G$ and the
graph $\gr(D\varphi)^G$ in the twisted product $\bar{V} \times
V$. Indeed, to see this, choose basis $\{ e_i \}$ for
$\gr(D\varphi)^G \subset \bar{\R}^{2n} \times \R^{2n}$ and let
$f_i \in N$ such that $\omega_0 (e_i, f_j)=\delta_{ij}$. The averaging
argument
\[
\delta_{ij}=\frac{1}{\vert G \vert} \sum_{g \in G} \omega_0 (ge_i,
gf_j) 
= \omega_0 (e_i , \frac{1}{\vert G \vert} \sum_{g \in G} gf_j)
\]
shows that $N^G$ is Lagrangian complement for $\gr(D\varphi)^G$ and as
well as for the diagonal $\triangle^G$ by symmetry.

Let $f \colon \triangle \rightarrow \R$ be a generating function for
$\varphi$ with respect to $\triangle \times N$. By construction,
$df \colon \triangle \rightarrow N$ is $G$-equivariant and the
restriction $df^G \colon \triangle^G \rightarrow N^G$ coincides with
the graph of $\varphi^G$ in $\triangle^G \times N^G$. Hence, the
restriction $f^G \colon \triangle^G \rightarrow \R$ is a generating
function for $\varphi^G$ with respect to $\triangle^G \times N^G$. Note
that since $df$ is $G$-equivariant, every primitive $f$ is
$G$-invariant and $df$ agrees with $d(f^G)$ on $\triangle^G$.

Using the generating functions $f$ and $f^G$ chosen above and the
isomorphism between the Floer homology and the generating function homology
(see \cite{Vi}) we pass to the local Morse homology:
\[
\HF_{*}(\varphi^G, x; \F_p)=\HM_*(f^G; \F_p)\textrm{ and }
\HF_{*}(\varphi, x; \F_p)=\HM_*(f; \F_p).
\]
Note that these isomorphisms are up to a shift in degree. Now the problem
reduces to showing that
\[
\dim \HM_*(f^G; \F_p) \leq \dim \HM_*(f; \F_p). 
\]

As the last step, we choose a $G$-invariant Gromoll--Meyer pair
$(U, U_{-})$ for $f$ as in \cite[App.\ B]{HHM}. (Actually, a stronger
result is proven in \cite{HHM}, but only for a cyclic group
$G$. However, the proof of the existence of an invariant
Gromoll--Meyer pair word-by-word extends to any finite group.) To
finish the proof observe that the fixed point set $(U^{G}, U_{-}^{G})$
is a Gromoll--Meyer pair for $f^G$. Theorem \ref{thm:2} follows now
from the classical Smith inequality, \eqref{eq:Smith-top}, applied to
$(U, U_{-})$:
\[
\dim \underbrace{\H_*(U^{G}, U_{-}^{G}; \F_p)}_{\HM_*(f^G;
  \F_p)}
\leq \dim \underbrace{ \H_*(U, U_{-}; \F_p)}_{\HM_*(f; \F_p)}. 
\]
\end{proof}

\begin{proof}[Proof of Theorem \ref{thm:2}] This theorem is a formal
  consequence of Theorem \ref{thm:2g}. As in \cite{He}, we use the
  symplectic version of Dold's trick discussed in Section
  \ref{sec:eq-str}. Namely, consider the germ of the
  symplectomorphism
  \begin{equation}
    \label{eq:varphi_p}
\varphi_p\colon (x_1,\ldots,x_p)\mapsto
\big(\varphi(x_2),\ldots,\varphi(x_p),\varphi(x_1)\big)
\end{equation}
of the $p$-fold product $M^p=M\times \ldots \times M$ near the fixed
point. This germ is Hamiltonian and it commutes with the symplectic
$G=\Z/p\Z$-action generated by the cyclic permutation
$$
(x_1,\ldots,x_p)\mapsto (x_2,\ldots,x_p,x_1),
$$
and $\bx:=(x,\ldots,x)$ is an isolated fixed point of $\varphi_p$
whenever $x$ is an isolated fixed point of $\varphi$. The fixed point
set $(M^p)^G$ is the multi-diagonal and the restriction $\varphi_p^G$
is just $\varphi$. Furthermore, the local Floer homology
$\HF_*(\varphi_p,\bx)$ is naturally isomorphic to the local Floer
homology $\HF_*(\varphi^p,x^p)$. This is a local
  version of a result proved in \cite{LL}, where, in particular, an
  isomorphism between the (global) Floer homology of the
  symplectomorphism $\varphi_p$ induced by $\varphi$ via
  \eqref{eq:varphi_p} and the Floer homology of the iterate
  $\varphi^p$ is constructed. (In fact, the result in \cite{LL}
  concerns Lagrangian correspondences and is much more general.) The
Smith inequality, \eqref{eq:smith}, follows now from
Theorem~\ref{thm:2g}.
\end{proof}

\begin{Remark}
  In the standard modern proof of the Smith inequality,
  \eqref{eq:Smith-top}, one obtains the inequality as an immediate
  consequence of the Borel localization theorem for $\Z_p$-actions
  (see, e.g., \cite{Bo}) and one can also expect a version of this
  theorem to hold for the local or filtered, in the aspherical case,
  Floer homology; cf.\ \cite{SZ}. The argument above falls just a
  little bit short from establishing such a theorem in the local
  case. The missing part is an identification of the
  $\Z_p$-equivariant generating function homology and the
  $\Z_p$-equivariant Floer homology -- an equivariant analog of a
  result from \cite{Vi}.

  On other hand, some other refinements of \eqref{eq:Smith-top} do not
  seem to have obvious Floer theoretic analogs. For instance,
 \cite[Section 3, Thm.\ 4.1]{Br} replaces the rank of the total homology by the
  sum of dimensions for degrees above a fixed treshold. We do not see
  how to extend this result to the Floer theoretic setting. (The most
  naive attempts break down already for a strongly non-degenerate
  orbit.)
\end{Remark}

\begin{Remark}
  The method used in this section to prove Theorem \ref{thm:2} can be
  applied whenever the (iterated) Floer homology can be identified
  with the homology of a sufficiently nice topological $\Z_p$-space
  and this identification behaves well under iteration. In particular,
  one should be able to use it to establish the Smith inequality for
  the filtered Floer homology of Hamiltonians on the tori or cotangent
  bundles. However, some parts of the argument require modifications
  and we omit the details; for the filtered version of the main
  theorem of \cite{ShZh} holds in much greater generality. Note in
  this connection that it would be interesting to see if the results
  of \cite{Se,ShZh} have purely equivariant generalizations along the
  lines of Theorem~\ref{thm:2g}.
\end{Remark}

\end{document}